\documentclass{llncs}

\usepackage{amsmath,amssymb,graphicx}

\newcommand{\Kucera}{Ku$\mathrm{\check{c}}$era}
\newcommand{\cs}{\ensuremath{\mathrm{\Omega}}}

\newcommand{\integral}[2]{\int\!{#1}\,\mathrm{d}{#2}}

\newtheorem{lem}[theorem]{Lemma}

\begin{document}

\title{A constructive version of Birkhoff's ergodic theorem for Martin-L\"of random points}

\author{Laurent Bienvenu\inst{1}
\and Adam Day\inst{2}
\and Mathieu Hoyrup \inst{3}
\and  Ilya Mezhirov \inst{4}
\and Alexander Shen \inst{5}\thanks{Supported by ANR Sycomore, NAFIT ANR-08-EMER-008-01, RFBR~09-01-00709-a grants and Shapiro visitors program at Penn State University.}
}
\institute{LIAFA, CNRS \& Universit{\'e} de Paris 7, France, {\tt laurent.bienvenu@liafa.jussieu.fr}\\
\and
LORIA, INRIA Nancy, France,  {\tt mathieu.hoyrup@loria.fr}\\
\and 
Victoria University of Wellington, New Zealand, {\tt adam.day@msor.vuw.ac.nz}\\
\and
Technical University of Kaiserslautern, {\tt mezhirov@gmail.com} \\
\and
LIF, CNRS \& Universit{\'e} d'Aix-Marseille 1, France, {\tt alexander.shen@lif.univ-mrs.fr}
}

\maketitle

\begin{abstract}

We prove the effective version of Birkhoff's ergodic theorem for Martin-L\"of random points and effectively open sets, improving the results previously obtained in this direction (in particular those of V.~Vyugin, Nandakumar and Hoyrup, Rojas). The proof consists of two steps. First, we prove a generalization of \Kucera's theorem, which is a particular case of effective ergodic theorem: a trajectory of a computable ergodic mapping that starts from a random point cannot remain inside an effectively open set of measure less than~$1$. Second, we show that the full statement of the effective ergodic theorem can be reduced to this special case.  Both steps use the statement of classical ergodic theorem but not its usual classical proof. Therefore, we get a new simple proof of the effective ergodic theorem (with weaker assumptions than before). 
\medskip

This result was recently obtained independently by Franklin, Greenberg, Miller and Ng.

\end{abstract}

\section{Introduction}

The classical setting for the ergodic theorem is as follows. Let $X$ be a space with a probability measure $\mu$ on it, and let $T\colon X\to X$ be a measure-preserving transformation. Let $f$ be a real-valued integrable function on $X$. Birkhoff's ergodic theorem (see for example~\cite{Shiryaev1996}) says that the average value 
	$$
\frac{f(x)+f(T(x))+ f(T(T(x)))+\ldots+f(T^{(n-1)}(x))}{n}
	$$
has a limit (as $n\to\infty$) for all $x$ except for some null set, and this limit (the ``time-average'') equals the ``space average'' $\integral{f(x)}{\mu(x)}$ if the transformation $T$ is ergodic (i.e., has no non-trivial invariant subsets).

The classical example of an ergodic transformation is the left shift on Cantor space $\cs$ (the set of infinite binary sequences, also denoted by $2^\mathbb{N}$ or $2^\omega$):
        \[
\sigma\big(\omega(0)\omega(1)\omega(2)\ldots\big)=\omega(1)\omega(2)\omega(3)\ldots
        \]
The left shift preserves Lebesgue measure (a.k.a.\ uniform measure)~$\mu$ on $\cs$ and is ergodic. Therefore, the time and space averages coincide for almost every starting point~$\omega$. For a special case where $f$ is an indicator function of some (measurable) set $A$, we conclude that almost surely (for all $\omega$ outside some null set) the fraction of terms in the sequence
        \[
\omega, \sigma(\omega), \sigma(\sigma(\omega)),\ldots
        \]
that are inside $A$, converges to the measure of $A$.

\medskip
It is natural to ask whether Birkhoff ergodic theorem has an effective version for individual points saying that for a Martin-L\"of random starting point the time average coincides with the space average (under some effectivity assumptions for the space and the transformation). This question was posed by van Lambalgen~\cite{vanLambalgen1987} and answered by Vyugin \cite{Vyugin1997} who proved this statement for the case of computable function~$f$ (he also proved the convergence result for non-ergodic transformations). The result was later extended to larger classes of functions \cite{Nan08,HoyrupR2009b}. However, we cannot directly apply these results to an indicator function of an effectively open set (recall that an open set~$U$ is effectively open if there is a computably enumerable set~$S$ of finite strings such that~$U$ consists exactly of the infinite sequences having a prefix in~$S$). Indeed, the characteristic function of such a set is not computable (it is only lower semicomputable, i.e., it is the limit of a non-decreasing sequence of computable functions). So for effectively open sets (and lower semicomputable functions) the question remained open.\footnote{It was proved in \cite{HoyrupR2009b} that the result holds for any effectively open set whose measure is computable.}

In this paper we answer this question and show that effective ergodic theorem remains true for effectively open sets and lower semicomputable functions (Section~\ref{ergodic-generalization}). The proof goes in several steps.

First, in Section~\ref{sec:variations-kucera} we consider the following corollary of Birkhoff's ergodic theorem: if $A$ has positive measure, for almost every starting point at least one element of the trajectory belongs to $A$. Switching to complements: if $A$ has measure less than $1$, then (almost surely) some points in the trajectory are outside~$A$. An effective version of this statement (for effectively open sets of measure less than~$1$ and left shifts in Cantor space) was proved by \Kucera~\cite{Kucera1985}. We reproduce \Kucera's proof and prove several similar statements. (Most of them are consequences of the general results of Section~\ref{ergodic-generalization}, so the direct proofs are redundant, but they are nice and simple.)

Then in Section~\ref{ergodic-generalization} we consider the general effective ergodic theorem. In Section~\ref{ergodic-generalization-kucera} we prove a general version of \Kucera's theorem for computable ergodic transformations of Cantor space. Then (Section~\ref{ergodic-generalization-birkhoff}) we show how the effective version of ergodic theorem for effectively open sets and lower semicomputable functions can be reduced to classical Birkhoff's ergodic theorem and the general version of \Kucera's theorem proved in Section~\ref{ergodic-generalization-kucera}. Finally, we outline the generalization of these results to other probability spaces (Section~\ref{subsec:cps}).

In Section~\ref{sec:gen-lambalgen} we use the results of Section~\ref{sec:variations-kucera} to provide a generalized version of van Lambalgen's theorem (generalizing an earlier result of Miyabe). 

The results of Sections~\ref{sec:variations-kucera} and~\ref{sec:gen-lambalgen} were presented at the Computability in Europe conference (and published in its proceedings~\cite{BienvenuDMS10}). The improvement in this paper is Theorem~\ref{th:effective-birkhoff}, showing that one can go further and reduce the general effective version of Birkhoff's ergodic theorem for effectively open sets to this special case. This last result was obtained independently in~\cite{FranklinGMN}.

\section{Variations of \Kucera's theorem}\label{sec:variations-kucera}

In this section, we prove several variants of \Kucera's theorem. Let us first recall the original version proved in \cite{Kucera1985}. Let $\sigma$ be the left shift in Cantor space (i.e., an ergodic transformation of this space equipped with uniform measure). 

\begin{theorem} \label{thm:uni-shift}
If $A$ is an effectively open subset of the Cantor space of measure less than~$1$, then for every Martin-L\"of random sequence $\omega$ at least one of its tails   $\omega$, $\sigma(\omega)$, $\sigma(\sigma(\omega))$,\ldots does not belong to $A$.
\end{theorem}

Recalling the definition of Martin-L\"of randomness (a sequence is random if it is outside any effectively null set) we can reformulate \Kucera's theorem as follows: 
\begin{quote}
\emph{Let $A$ be an effectively open set of measure less than $1$. Consider the set $A^*$ of all sequences $\omega$ such that every tail $\sigma^{(n)}(\omega)$ belongs to $A$. Then $A^*$ is an effectively null set.}
\end{quote}

Before presenting the proof, let us mention an interpretation of this result. Recall that the universal Martin-L\"of test is a computable sequence $U_1, U_2,\ldots$ of effectively open sets such that $\mu(U_i)\le 1/2^i$ and the intersection $\cap_i U_i$ is the maximal effectively null set, i.e., the set of all non-random sequences. \Kucera's theorem shows that randomness can be (in a paradoxical way) characterized by $U_1$ alone: a sequence is non-random if and only if all its tails belong to $U_1$. (In one direction it is \Kucera's theorem, in the other direction we need to note that a tail of a non-random sequence is non-random.)

\begin{proof} [of \Kucera's theorem]
 We start with the following observation: it is enough to show that for every interval $I$, we can uniformly construct an effectively open set $J\subset I$ that contains $I\cap A^*$ and such that  $\mu(J)\le r \mu(I)$ for some fixed $r<1$ (here we call an \emph{interval} any set of type $x\cs$, where~$x$ is some finite string, i.e., the set of infinite binary sequences that start with~$x$). Then we represent the effectively open set $A$ of measure $r<1$ as a union of disjoint intervals $I_1,I_2,\ldots$, construct the sets $J_i$ for every $I_i$ and note that the union $A_1$ of all $J_i$ is an effectively open set that contains $A^*$ and has measure $r^2$ or less. Splitting $A_1$ into disjoint intervals and repeating this argument, we get a set $A_2$ of measure at most $r^3$, etc. In this way we get a effectively open cover for $A^*$ of arbitrarily small measure, so $A^*$ is an effectively null set.

It remains to show how to find $J$ given $I$. The interval $I$ consists of all sequences that start with some fixed prefix $x$, i.e., $I=x\cs$. Since sequences in $A^*$ have all their tails in $A$, the intersection $I\cap A^*$ is contained in $xA$, and the latter set has measure $r \mu(I)$ (where $r=\mu(A)$). \qed
\end{proof}

Note that this proof also shows the following: suppose~$A$ is an effectively open set of measure less than~$1$, and $A$ can be written as a disjoint union of intervals $A=x_1\cs \cup x_2\cs \cup \ldots$. Let $\omega$ be an infinite sequence that can be written as $\omega=w_1w_2w_3 \ldots$ where for all~$i$, $w_i=x_j$ for some~$j$. Then $\omega$ is not random. (If $A$ contains all non-random sequences, the reverse implication is also true, and we get yet another criterion of randomness.)

\subsection{Effective Kolmogorov 0-1 law}

Trying to find characterizations of randomness similar to \Kucera's theorem, one may look at Kolmogorov's $0$-$1$-law. It says that any measurable subset~$A$ of the Cantor space that is stable under finite changes of bits (i.e. if $\omega \in A$ and $\omega'$ is equal to $\omega$ up to a finite change of bits, then $\omega' \in A$) has measure $0$ or $1$. It can be reformulated as follows: let $A$ be a (measurable) set of measure less than~$1$. Consider the set $A^*$ defined as follows: $\omega\in A^*$ if and only if all sequences that are obtained from $\omega$ by changing finitely many terms, belong to $A$. Then $A^*$ has measure zero (indeed, $A^*$ is stable under finite changes and cannot have measure~$1$). Note also that we may assume without loss of generality that $A$ is open (replacing it by an open cover of measure less than~$1$).

A natural effective version of Kolmogorov's $0$-$1$-law can then be formulated as follows. (In fact, this statement was considered and proved by \Kucera\ but was not explicitly mentioned in \cite{Kucera1985}.)

\begin{theorem} \label{thm:finite-change}
Let $A$ be an effectively open set of measure $r<1$. Consider the set $A^*$ of all sequences that belong to $A$ and remain in $A$ after changing finitely many terms. Then $A^*$ is an effectively null set.
\end{theorem}
(As we have seen, the last two sentences can be replaced by the following claim: \emph{any Martin-L\"of random sequence can be moved outside $A$ by changing finitely many terms.})

\begin{proof}To prove this effective version of the $0$-$1$-law, consider any interval $I$. As before, we want to find an effectively open set $U\subset I$ that contains $A^*\cap I$ and has measure at most $r\mu(I)$. Let $x$ be the prefix that defines $I$, i.e., $I=x\cs$. For every string $y$ of the same length as~$x$, consider the set 
	$
A_y=\{ \omega \mid y \omega\in A\}.
	$
It is easy to see that the average measure of $A_y$ (over all $y$ of a given length) equals $\mu(A)=r$. Therefore, the set
	$
B=\bigcap_{y} A_y
	$
(which is effectively open as an intersection of an effectively defined finite family of open sets) has measure at most~$r$. Now take $U=xB$. Let us show that $U$ is as wanted. First,~$U$ is an effectively open set, contained in~$I$, and of measure $r \mu(I)$. Also, it contains every element of $A^*\cap I$. Indeed, if $\alpha \in A^* \cap I$, $x$ is a prefix of $\alpha$, so one can write $\alpha=x \beta$. Since $\alpha \in A^*$, any finite variation of $\alpha$ is in~$A$, so for all~$y$ of the same length as~$x$, $y\beta \in A$. Therefore, $\beta$ is in all~$A_y$, and therefore is in $B$. Since $\alpha=x\beta$, it follows that $\alpha$ is in $xB=U$. \qed
\end{proof} 	

\subsection{Adding prefixes}
 
We have considered left shifts (deletion of prefixes) and finite changes. Another natural transformation is the \emph{addition of finite prefixes}. It turns out that a similar result can be proven in this case (although the proof becomes a bit more difficult).

\begin{theorem}\label{thm:add-prefix}
Let $A$ be an effectively open set of measure $r<1$. Let $A^*$ be the set of all sequences $\omega$ such that $x\omega\in A$ for every  binary string $x$. Then $A^*$ is an effectively null set. In other words, for every Martin-L\"of random sequence $\omega$ there exists a string $x$ such that $x\omega\notin A$.
\end{theorem}

\begin{proof}To prove this statement, consider again some interval $I=x\cs$. We want to cover $A^*\cap I$ by an effectively open set of measure $r\mu(I)$. (In fact, we get a cover of measure $s\mu(I)$ for some constant $s\in(r,1)$, but this is enough.) Consider some string $z$. We know that the density of $A^*$ in $I$ does not exceed the density of $A$ in $zI=zx\cs$. Indeed, $x\omega\in A^*$ implies $zx\omega\in A$ by definition of $A^*$. 

Moreover, for any finite number of strings $z_1,\ldots,z_k$ the set $A^*$ is contained in the intersection of sets $\{\omega\mid z_i \omega\in A\}$, and the density of $A^*$ in $I$ is bounded by the minimal (over $i$) density of $A$ in $z_i I=z_ix\cs$.

Now let us choose $z_1,\ldots,z_k$ in such a way that the intervals $z_ix\cs$ are disjoint and cover $\cs$ except for a set of small measure. This is possible for the same reason as in a classic argument that explains why the Cantor set in $[0,1]$ has zero measure. We start, say, with $z_1=\Lambda$ and get the first interval $x\cs$. The rest of $\cs$ can be represented as a union of disjoint  intervals, and inside each interval $u\cs$ we select a subinterval $ux\cs$ thus multiplying the size of the remaining set by $(1-2^{-|x|})$. Since this procedure can be iterated indefinitely, we can make the rest as small as needed.

Then we note that the density of $A$ in the union of disjoint intervals (and this density is close to $r$ if the union covers $\cs$ almost entirely) is greater than or equal to the density of $A$ in one of the intervals, so the intersection (an effectively open set) has density at most $s$ for some constant $s\in(r,1)$, as we have claimed. (We need to use the intersection and not only one of the sets since our construction should be effective even when we do not know for which interval the density is minimal.) \qed
\end{proof}

\subsection{Bidirectional sequences and shifts}\label{subsec:bidirectional}

Recall the initial discussion in terms of ergodic theory. In this setting it is more
natural to consider bi-infinite binary sequences, i.e., mappings of
type $\mathbb{Z}\to\mathbb{B}=\{0,1\}$; the uniform measure 
$\mu$ can be naturally defined on this space, too. On this space the transformation~$T$ corresponding to the shift to the left is reversible: any sequence can be shifted left or right.

The result of Theorem~\ref{thm:uni-shift} remains true in this setting.

\begin{theorem} \label{thm:bi-shift}
Let $A$ be an effectively open set of $\mathbb{B}^\mathbb{Z}$, of measure $r<1$. The set~$A^*$ of all sequences that remain in~$A$ after any arbitrary shift (any distance in any direction) is an effectively null set.
\end{theorem}

To prove this statement, consider any $s\in(r, 1)$.
As usual, it is enough to find (effectively) for every interval $I_x$ an effectively open subset of $I_x$ that contains $A^*\cap I_x$ and has measure at
most $s\mu(I_x)$. Here $x$ is a finite partial function from $\mathbb{Z}$ to $\mathbb{B}$ and $I_x$ is the set of all its extensions. (One may assume that $x$ is contiguous, since every other interval is a finite union of disjoint contiguous intervals, but this is not important for us.) Then we may iterate this construction, replacing each interval of an effectively open set by an open set inside this interval, and so on until the total measure ($s^k$, where $k$ is the number of iterations) becomes smaller than any given $\varepsilon>0$.

Assume that some $I_x$ is given. Note that $A^*$ is covered by every shift of $A$, so any intersection of $I_x$ with a finite collection of shifted versions of~$A$ (i.e., sets of type $T^n(A)$ for $n \in \mathbb{Z}$) is a cover for $I_x \cap A^*$. It remains to show that the intersection of properly chosen shifts of~$A$ has density at most $s$ inside $I_x$. To estimate the measure of the intersection, it is enough to consider the minimum of measures, and the minimum can be estimated by estimating the average measure.

More formally, we first note that by reversibility of the shift and the invariance of the measure, we have
\[
\mu \big(I_x \cap T^{-n}(A) \big)= \mu\big(A \cap T^{n}(I_x) \big)
\]
for all~$n$. Then we prove the following lemma:

\begin{lem}\label{lem:average}
 Let $J_{1},\ldots,J_{k}$ be independent intervals of the same measure $d$ corresponding to disjoint functions $x_1,\ldots,x_k$ of the same length. Then the average of the numbers
	$$
\mu(A\cap J_{1}),\ldots,\mu(A\cap J_{k})
        $$
does not exceed~$s d$ if~$k$ is large enough. Moreover such a~$k$ can be found effectively.
\end{lem}

\begin{proof}[of Lemma~\ref{lem:average}]
The average equals
	$$
\frac{1}{k}\sum_i\mathsf{E} (\chi_A \cdot \chi_i) 
	$$
where $\chi_A$ is the indicator function of $A$ and $\chi_i$ is the 
indicator function of $J_i$. Rewrite this as
        $$
\mathsf{E} \left(\chi_A \cdot\frac{1}{k}\sum_i \chi_i\right),
	$$
and note that 
	$$
\frac{1}{k}\sum_i \chi_i
	$$
is the frequency of successes in~$k$ independent trials with individual probability~$d$. (Since the functions $x_i$ are disjoint, the corresponding intervals $J_i$ are independent events.) This frequency (as a function on the bi-infinite Cantor space $\mathbb{B}^\mathbb{Z}$) is close to $d$ everywhere except for a set of small measure (by the central limit theorem; in fact Chebyshev's inequality is enough). The discrepancy and the measure of this exceptional set can be made as small as needed using a large~$k$, and the difference is then covered by the gap between~$r$ and $s$. This ends the proof of the lemma. \\

Now, given an interval $I_x$, we cover $I_x \cap A^*$ as follows. First, we take a integer~$N$ larger than the size of the interval $I_x$. The intervals
\[
T^{N}(I_x), T^{2N}(I_x), T^{3N}(I_x), \ldots
\]
are independent and have the same measure as $I_x$, so we can apply the above lemma and effectively find a $k$ such that the average of 
	$$
\mu(A\cap T^{N}(I_x)),\ldots,\mu(A\cap T^{kN}(I_x))
        $$
does not exceed $s \mu(I_x)$. This means that for some $i \leq k$ one has
\[
\mu(I_x \cap T^{-iN}(A)) = \mu(A\cap T^{iN}(I_x)) \leq  s \mu(I_x)    
\]
Therefore, $I_x \cap \bigcap_{i \leq k} T^{-iN}(A)$ is an effectively open cover of $A^*$ of measure at most $s \mu(I_x)$. \qed
\end{proof}

The statement can be strengthened: we can replace all shifts by any infinite enumerable family of shifts.

\begin{theorem}\label{thm:bi-shift-re-set}
Let $A$ be an effectively open set \textup(of bi-infinite sequences\textup) of measure $\alpha<1$. Let $S$ be an computably enumerable infinite set of integers. Then the set
	$$
A^*=\{\omega\mid \text{$\omega$ remains in $A$ after shift by $s$, for every $s\in S$}\}
	$$
is an effectively null set.
\end{theorem}

(Reformulation: \emph{let $A$ be an effectively open set of measure less than~$1$; let $S$ be an infinite computably enumerable set of integers; let $\alpha$ be a Martin-L\"of random bi-infinite sequences. Then there exists $s\in S$ such that the $s$-shift of $\omega$ is not in~$A$}.) 

\begin{proof}
The proof remains the same: indeed, having infinitely many shifts, we can choose as many disjoint shifts of a given interval as we want.
\qed
\end{proof}

The argument used to prove Theorem~\ref{thm:bi-shift} (and Theorem~\ref{thm:bi-shift-re-set}) is more complicated than the previous ones (that do not refer to the central limit theorem): previously we were able to use disjoint intervals instead of independent ones. In fact the results about shifts in unidirectional sequences (both) are corollaries of the last statement. Indeed, let $A$ be an effectively open set of right-infinite sequences of measure less than $1$. Let $\omega$ be a right-infinite Martin-L\"of random sequence. Then it is a part of a bi-infinite random sequence $\bar \omega$ (one may use, e.g., van Lambalgen's theorem~\cite{vanLambalgen1987} on the random pairs, see  Section~\ref{sec:gen-lambalgen} for a precise statement). So there is a right shift that moves $\bar \omega$ outside $\bar A$, and also a left shift with the same property (here by $\bar A$ we denote the set of bi-infinite sequences whose right halves belong to $A$).

\section{A generalization to all ergodic transformations}\label{ergodic-generalization}

\subsection{Generalizing \Kucera's theorem}\label{ergodic-generalization-kucera}

First let us recall the notion of a computable transformation of the Cantor space~$\cs$. Consider a Turing machine with a read-only input tape and write-only output tape (where head prints a bit and moves to the next blank position). Such a machine determines a computable mapping of $\cs$ into the space of all finite and infinite binary sequences. Restricting this mapping to the inputs where the output sequence is infinite, we get a (partial) computable mapping from $\cs$ into~$\cs$.

\begin{theorem}\label{thm:generalized-kucera}
Let $\mu$ be a computable measure on~$\cs$. Let $T: \cs \rightarrow \cs$ be a partial computable, almost everywhere defined, measure-preserving, ergodic transformation of $\cs$. Let~$A$ be an effectively open subset of $\cs$ of measure less than~$1$. Let $A^*$ be the set of points~$x \in \cs$ such that $T^n(x) \in A$ for all $n\geq 0$. Then, $A^*$ is an effectively null set.
\end{theorem}

\begin{proof}
Let $r$ be a real number such that $\mu(A) < r < 1$. As before, given an interval~$I$, we want to (effectively) find an~$n$ such that $I \cap \bigcap_{i \leq n} T^{-i}(A)$ has measure at most $r \mu(I)$. This gives us an effectively open cover of $A^*\cap I$ having measure at most $r\mu(I)$; iterating this process, we conclude that $A^*$ is an effectively null set.

(A technical clarification is needed here. If we consider $T$ only on inputs where the output sequence is infinite, the set $T^{-1}(A)$ (and in general $T^{-i}(A)$) may no longer be open in~$\cs$. But since $T$ is almost everywhere defined, we may extend $T$ to the space $\widehat{\cs}$ of infinite \emph{and} finite sequences in a natural way and get an effectively open cover of the same measure.)

To estimate  $\mu(I \cap \bigcap_{i \leq n} T^{-i}(A))$, we note that it does not exceed the minimal value of $\mu(I \cap T^{-i}(A))$, which in its turn does not exceed the average (over $i\le n$) of $\mu(I\cap T^{-i}(A))$. This average,
\[
\tfrac{1}{n+1} \left[\mu(I\cap A)+\mu(I\cap T^{-1}(A))+\ldots+\mu(I\cap T^{-n}(A))\right]\eqno(*)
\]
can be rewritten as
\[
\tfrac{1}{n+1} \left[\mu(T^{-n}(I)\cap T^{-n}(A))+\mu(T^{-(n-1)}(I)\cap T^{-n}(A))+\ldots+\mu(I\cap T^{-n}(A))\right]
\]
since $T$ is measure preserving. The latter expression is the inner product of the indicator function of $T^{-n}(A)$ and the average $a_n=(\chi_0+\ldots+\chi_n)/(n+1)$, where $\chi_i$ is the indicator function of $T^{-i}(I)$.

As $n\to\infty$, the average $a_n$ converges in $L_2$ to the constant function $\mu(I)$, due to von Neumann's mean ergodic theorem. By the Cauchy--Schwarz inequality, this implies that the scalar product converges to $\mu(A)\mu(I)$ and therefore does not exceed $r\mu(I)$ for $n$ large enough.

It remains to (effectively) find a value of~$n$ for which the $L_2$-distance between $a_n$ and the constant $\mu(I)$ is small. Note that for all~$i$ the set $T^{-i}(I)$ is an effectively open set of measure $\mu(I)$ (recall that  $T$ is measure preserving), and $\mu(I)$ is computable since $\mu$ is a computable measure. Therefore, for any~$i$ and $\varepsilon>0$, one can uniformly approximate $T^{-i}(I)$ by its subset $U$ that is a finite union of intervals such that $\mu(T^{-i}(I) \setminus U) < \varepsilon$. This means that the $L_2$-distance between $a_n$ and the constant function $\mu(I)$ can be computed effectively, and we can wait until we find a term with any precision needed. In particular, we can effectively find an~$n$ such that the average $(*)$ is less than~$r$. By the above discussion, we then have $\mu(I \cap \bigcap_{i \leq n} T^{-i}(A))<r\mu(I)$, as needed.
\qed

\end{proof}

Now we get all the theorems of Section~\ref{sec:variations-kucera} (except for Theorem~\ref{thm:bi-shift-re-set}) as corollaries: the effective ergodic theorem for the bidirectional shift (Theorem~\ref{thm:bi-shift}) immediately follows as the bidirectional shift is clearly computable, measure-preserving and ergodic. Remark: technically we proved Theorem~\ref{thm:generalized-kucera} only for the Cantor space $\cs$, but the space of functions $\mathbb{Z} \rightarrow \mathbb{B}$ on which the bidirectional shift is defined, is computably isomorphic to $\cs$. By this we mean that there exists a computable measure preserving bijection from one space to another; for example, one could represent a two-directional sequence $\ldots \omega(-2)\omega(-1)\omega(0)\omega(1)\omega(2)\ldots$ by a one-directional sequence $\omega(0)\omega(-1)\omega(1)\omega(-2)\omega(2)\ldots$, and under this representation we can therefore represent the bidirectional shift as a measure preserving map from $\cs$ to itself. \\

Recalling the discussion in Section~\ref{subsec:bidirectional}, we see also that one can derive both Theorem~\ref{thm:uni-shift} (\Kucera's theorem for deletion of finite prefixes) and Theorem~\ref{thm:add-prefix} (addition of finite prefixes) from Theorem~\ref{thm:generalized-kucera}.

It turns out that even Theorem~\ref{thm:finite-change} (finite change of bits) can be proven in this way. Indeed, let us consider the map $F$ defined on $\cs$ by:  
\[
 F(1^n0\omega)=0^n1\omega~ ~ \text{for all $n$}, ~ ~  \text{and} ~ ~ F(11111 \ldots)=00000 \ldots
\]
\label{finite-change-function}%
($F$ adds 1 to the sequence in the dyadic sense). It is clear that~$F$ is computable and measure-preserving. That it is ergodic comes from Kolmogorov's 0-1 law, together with the observation that any two binary sequences $\omega, \omega'$ that agree on all but finitely many bits are in the same orbit:   $\omega'=F^n(\omega)$ for some $n \in \mathbb{Z}$.  The reverse is also true except for the case when sequences have finitely many zeros or finitely many ones. This cannot happen for a random sequence, so this exceptional case does not prevent us to derive Theorem~\ref{thm:finite-change} from Theorem~\ref{thm:generalized-kucera}.

\begin{remark}
Theorem~\ref{thm:bi-shift-re-set} asserts that given a random $\omega$, and an effectively open set $U$ of measure less than~$1$, there exists an~$n$ such that $T^n(\omega) \notin U$ (where $T$ is the shift in the space of bidirectional sequences), and that moreover~$n$ can be found in a computable enumerable set fixed in advance. This of course still holds for the unidirectional shift on $\cs$, but this does not hold for all ergodic maps. Indeed, this fact is related 
to the so-called \emph{strong mixing property} of the shift, which not all ergodic maps have. For example, a rotation of the circle by a computable irrational angle $\alpha$ (i.e., a mapping $x\mapsto x+\alpha\bmod 1$ on $\cs$ seen as the interval~$[0,1]$) is a computable ergodic map that does not have this property, and it is easy to construct a counterexample to the claim of Theorem~\ref{thm:bi-shift-re-set} for that particular map. 
\end{remark}

\subsection{An effective version of Birkhoff's ergodic theorem}\label{ergodic-generalization-birkhoff}

The generalization of \Kucera's theorem we proved in the previous section (Theorem~\ref{thm:generalized-kucera}) is only a weak form of ergodic theorem. It asserts that under the action of a computable ergodic map, the orbit of a Martin-L\"of point will intersect any given effectively closed set of positive measure, but it does not say anything about the frequency. This is what we achieve with the next theorem.

\begin{theorem}\label{th:effective-birkhoff}
Let $\mu$ be a computable measure on $\cs$.
Let $T:\cs \rightarrow \cs$ be a computable almost everywhere defined $\mu$-preserving ergodic transformation. Let $U$ be an effectively open set. For every Martin-L\"of random point~$\omega$,
\[
\lim_{n\to\infty} \frac{1}{n}\sum_{k=0}^{n-1}\chi_U( T^k(\omega))=\mu(U).
\]
\end{theorem}

Note that the statement is symmetric, so the same is true for an effectively \emph{closed} set $C$.

\begin{proof}
Let $g_n(\omega) =\frac{1}{n}\sum_{k=0}^{n-1}\chi_U (T^k(\omega))$ be the frequency of $U$-elements among the first $n$ iterations of $\omega$. Let us first prove that $\limsup g_n(\omega)\le \mu(U)$. Then we show (see part (2) below) that $\liminf g_n(\omega)\ge \mu(U)$.

(1) Let  $r>\mu(U)$ be some rational number and let
	$$G_N=\{\omega\colon(\exists n\geq N)\, g_n(\omega)>r\}$$
be the set of points where some far enough frequency (average of at least $N$ terms) exceeds~$r$. The set $G_N$ is an effectively open set; indeed, the functions $g_n$ are lower semicomputable (uniformly in~$n$), hence the condition~$g_n(\omega)>r$ is enumerable. The sets $G_N$ form a decreasing sequence. We know by the classical Birkhoff's pointwise ergodic theorem that  $\mu(\bigcap_N G_N)=0$, since the sequence of functions $g_n$ converges to $\mu(U)<r$ $\mu$-almost everywhere. As a result, there exists~$N$ such that $\mu(G_N)<1$. We can thus apply Theorem~\ref{thm:generalized-kucera} to this $G_N$ and conclude  that for every Martin-L\"of random $\omega$ there exists $k$ such that $T^k(\omega)\notin G_N$. Hence $\limsup_n g_n(T^k(\omega))\le r$. Since a finite number of iterations does not change the $\limsup$, we conclude that $\limsup g_n(\omega)\le r$. The number $r$ was an arbitrary rational number greater than $\mu(U)$, so $\limsup g_n(\omega) \leq \mu(U)$. 

(2) We now prove that $\liminf g_n(\omega)\ge \mu(U)$. This in fact can be deduced from the first part of the proof. The set~$X$ is open, so it is a countable union of disjoint intervals. Taking a finite part of this countable union, we get an effectively closed set $C \subset U$ and can apply the previous statement to its complement. It says that the orbit of a Martin-L\"of random point $\omega$ will be in~$X'$ with frequency at least~$\mu(X')$ (the upper bound for the complement of~$C$ means a lower bound for~$C$). Since $\mu(C)$ can be arbitrarily close to $\mu(X)$, we conclude that $\liminf g_n(\omega)\ge \mu(U)$.\qed
\end{proof}

\begin{remark}
The inequality $\liminf g_n(\omega)\ge \mu(X)$ can actually be derived from the algorithmic version of Birkhoff's theorem proved by V'yugin \cite{Vyugin1997}, since $X'$ is open and closed set, but it is easier to refer to the first part of the proof. Note also that in this direction we do not need effectivity: $\liminf g_n(\omega) \ge \mu(X)$  for every open set $X$ and every Martin-L\"of random point $\omega$. Of course the other inequality generally fails for (non-effectively) open sets: indeed, the orbit of every point $\omega$ can be enclosed in a (non-effectively) open set of small measure.
\end{remark}

Theorem \ref{th:effective-birkhoff} extends to a larger class of sets in a straightforward way. We say that a set $A$ is effectively $\mu$-approximable if $\mu(A)=\sup\{\mu(F):F$ effectively closed and $F\subseteq A\}=\inf\{\mu(G):G$ effectively open and $A\subseteq G\}$. For instance, any $\mathrm{\Delta}^0_2$-set is effectively $\mu$-approximable.

\begin{corollary}
Let $X\subset \cs$ be an effectively $\mu$-approximable set. For every Martin-L\"of $\mu$-random $\omega$, $\lim \frac{1}{n}(\chi_X(\omega)+\ldots+\chi_X(T^{n-1}(\omega)))=\mu(X)$.
\end{corollary}
\begin{proof}
For every $\varepsilon>0$ we can apply Theorem~\ref{th:effective-birkhoff} to the upper and lower $\varepsilon$-approximations of $X$; the frequency for $X$ is between them.\qed
\end{proof}

Theorem \ref{th:effective-birkhoff} can also be extended a wider class of functions than characteristic functions of sets. 

\begin{theorem}\label{thm:effective-birkhoff}
Let $f:\cs \to[0,+\infty]$ be lower semicomputable. For every Martin-L\"of random $\omega$,
\[
\lim_{n\to\infty} \frac{1}{n}\sum_{k=0}^{n-1}f(T^k(\omega))=\integral{f}{\mu}.
\]
\end{theorem}

Note that we allow the integral to be infinite; in this case the sequence in the left-hand side has limit $+\infty$.

\begin{proof}
Let~$f$ be a lower semicomputable function with a finite integral. Let $f_n=\frac{1}{n}(f+\ldots+f\circ T^{n-1})$. Let $r>\integral{f}{\mu}$ be a rational number and 
           $$G_N=\{\omega\colon(\exists n\ge N)\, f_n(\omega)>r\}.$$
The set $G_N$ is an effective open set and $\mu(\bigcap_N G_N)=0$ as $f_n(\omega)\to \integral{f}{\mu}<r$ for $\mu$-almost every $\omega$ (by the classical version of Birkhoff's ergodic theorem). As a result, there exists $N$ such that $\mu(G_N)<1$. By Theorem~\ref{thm:generalized-kucera}, if $\omega$ is Martin-L\"of random then there exists $k$ such that $T^k(\omega)\notin G_N$. Hence $\limsup f_n(T^k(\omega))\leq r$, and $\limsup f_n(\omega)=\limsup f_n(T^k(\omega))\le r$. Since $r>\integral{f}{\mu}$ can be arbitrarily close to the integral, we have that $\limsup f_n(\omega)\leq \integral{f}{\mu}$.

It remains to prove that $\liminf f_n(\omega)\geq \integral{f}{\mu}$. This is true for every lower semicontinuous $f$. Indeed, consider some lower bound for $f$ that is a basic function (a linear combination of indicators of intervals). For these basic functions the statement of the theorem is true (as we already know), and their integrals can be arbitrarily close to $\integral{f}{\mu}$. (This argument works also for the case $\integral{f}{\mu}=+\infty$.)\qed
\end{proof}

Theorem~\ref{thm:effective-birkhoff} is, to the extent of our knowledge, the strongest form of effective ergodic theorem proven so far, in the case of an ergodic transformation. In particular, it strengthens the results that appeared in~\cite{Vyugin1997,Nan08,HoyrupR2009b} for ergodic measures. We will see in the next section that it can even be extended a bit further, namely to other spaces than~$\cs$ and to ergodic maps that are only ``weakly computable'' (in a sense which we will explain below). However, whether the Birkhoff averages of an effectively open set converge at all Martin-L\"of random points when the measure is \emph{not ergodic} remains an open problem (note that in the non-ergodic case, if the limit exists at a point, that limit is no longer the measure of the open set but depends on the particular point).
\smallskip

But let us mention first an interesting consequence of Theorem~\ref{thm:effective-birkhoff}. Recall that the randomness deficiency of a sequence~$\omega$ is defined as
$$
d_\mu(\omega)=\sup_n \{-\log \mu[\omega_0\ldots\omega_{n-1}]-K(\omega_0\ldots\omega_{n-1})\}
$$
where $K(w)$ is the (prefix) Kolmogorov complexity of $w$.

The following was proven by G\'acs~\cite{Gacs1980}: a sequence $\omega$ is Martin-L\"of random with respect to $\mu$ if and only if $d_\mu(\omega)$ is finite. Moreover, $t_\mu:=2^{d_\mu}$ is a universal randomness test in the sense that it is lower semicomputable, $\mu$-integrable, and for every lower semicomputable $\mu$-integrable $f:\cs\to[0,+\infty]$ there exists $c$ such that $f\leq ct_\mu$.

For a computable $\mu$-preserving mapping $T$ it is already known that if $\omega$ is Martin-L\"of random, then so are $T(\omega)$, $T^2(\omega)$, etc. Theorem \ref{thm:effective-birkhoff} applied to $t_\mu$ yields a stronger result for the case of ergodic $T$: not only the values $t_\mu(\omega)$, $t_\mu(T(\omega))$, $t_\mu(T^2(\omega))$, etc. are finite, but also their average is bounded. In this sense, the iterates of a random point are ``random in the average''. It is still an open problem whether this still holds in the non-ergodic case.

\subsection{A final generalization: computable probability spaces\\ and layerwise computable functions}\label{subsec:cps}
We now briefly present two ``orthogonal'' ways in which the previous results can be extended to other contexts. On the one hand, the algorithmic theory of randomness has been extended from the Cantor space to any computable metric space, where the computability of probability measures is now well understood. All the results presented above extend to such spaces. On the other hand, on the Cantor space as well as any computable metric space, the computability assumption on the mapping $T$ can be weakened into layerwise computability introduced in \cite{HoyupR2009c}. Intuitively, this weakening corresponds in analysis to replacing continuity with measurability.

The first generalization can be carried out in two ways: the proof on the Cantor space can generally be adapted to any computable probability space, or the isomorphism between such spaces (see \cite{HoyrupR2009}) can be used to transfer the result without proving it again. The second generalization is also rather direct: replacing computability notions with their ``layerwise'' counterparts generally leaves the proofs correct. Caution is sometimes needed and appropriate lemmas then have to be used (especially regarding composition of functions).

We now give a brief overview of the aforementioned concepts. More details can be found in \cite{Gac05,HoyrupR2009,HoyrupR2009b,BienvenuGHRS2011}.

The algorithmic theory of randomness has been extended from the Cantor space to any computable metric space, i.e. any separable metric space with a distinguished dense countable set on which the metric is computable. A computable probability space is such a space $X$, endowed with a computable Borel probability measure $\mu$. A universal Martin-L\"of test always exist on such spaces, and induces a canonical decomposition of the set of Martin-L\"of random points $\mathcal{R}^\mu=\bigcup_n \mathcal{R}^\mu_n$ with $\mathcal{R}^\mu_n\subseteq \mathcal{R}^\mu_{n+1}$ and $\mu(\mathcal{R}^\mu_n)>1-2^{-n}$ (namely, $\mathcal{R}^\mu_n$ is the complement in~$X$ of the $n$-th level of a universal $\mu$-Martin-L\"of test). Using this decomposition, one can weaken many computability notions, starting with the notion of a computable function: we say that a function $f:X\to Y$ (where $Y$ is a computable metric space) is \emph{$\mu$-layerwise computable} if it is computable on each $\mathcal{R}^\mu_n$ (uniformly in $n$)\footnote{When $X=Y=\cs$, it means that there is a Turing machine that on input $n$ and oracle $x\in\mathcal{R}^\mu_n$ progressively writes $f(x)$ on the output tape. The machine does not need to behave well when $x\notin \mathcal{R}^\mu_n$.}. Such a function may be discontinuous, but is still continuous on each $\mathcal{R}^\mu_n$, which is a totally disconnected set. It turns out that this notion admits a characterization in terms of effective measure theory.

Observe that $\mu$-layerwise computability of real-valued functions is closed under basic operations such as sum, product, multiplication by a computable real number, and absolute value. Composition does not automatically preserve layerwise computability without an assumption on the preservation of the measure. If $f:X\to[-\infty,+\infty]$ and $T:X\to X$ are $\mu$-layerwise computable and $T$ preserves $\mu$, then $f\circ T$ is $\mu$-layerwise computable. If, moreover, $f$ is bounded, then $\int\!{f}\,\mathrm{d}{\mu}$ is computable, uniformly in~$f$ and a bound on~$f$. In particular, $\Vert f\Vert_1$ and $\Vert f\Vert_2$ are computable.

The main reason for which layerwise computability fits well with Martin-L\"of randomness is that Martin-L\"of random points pass a class of tests that is wider than the usual Martin-L\"of tests: the tests that, on each $\mathcal{R}^\mu_k$, ``look like'' Martin-L\"of tests.

\begin{lem}\label{lem:lay-test}
Let $A_n\subseteq X$ be such that there exist uniformly effective open sets $U_{n,k}$ such that $A_n\cap\mathcal{R}^\mu_k=U_{n,k}\cap \mathcal{R}^\mu_k$. If $\mu(A_n)<2^{-n}$ for all $n$, then every $\mu$-random point is outside $\bigcap_n A_n$. Moreover there is $c$ such that $\mathcal{R}^\mu_n\cap A_{n+c}=\varnothing$ for all $n$.
\end{lem}
\begin{proof}
Let $V_n=U_{n,n}\cup (X\setminus \mathcal{R}^\mu_n)$: $V_n$ is a Martin-L\"of test and $A_n\subseteq V_n$.
\end{proof}

Let us show how to adapt a part of the proof of Theorem \ref{thm:generalized-kucera} to computable probability spaces and $\mu$-layerwise computable mappings.

\begin{theorem}\label{thm:generalized-generalized-kucera}
Let $(X,\mu)$ be a computable probability space. Let $T: X \rightarrow X$ be a $\mu$-layerwise computable, measure-preserving, ergodic transformation of $X$. Let~$A$ be an effectively open subset of $X$ of measure less than~$1$. For every $\mu$-random point $x$, there exists $n$ such that $T^n(x)\notin A$.
\end{theorem}

\begin{proof}[Sketch]
The proof is essentially the same as that of Theorem~\ref{thm:generalized-kucera}. The only differences are: adapting the notion of cylinder; using properties of layerwise computability; using Lemma \ref{lem:lay-test}.

\medskip
A computable probability space always admits a basis of metric balls with computable centers and radii, whose borders have null measure. These balls correspond in a sense to the cylinders of the Cantor space: for instance, their measures are computable. Let then $B=B(x,r)$ be a metric ball with computable center and radius, such that $\mu(\{y:d(x,y)=r\})=0$. Then $\mu(B)$ is computable, $\chi_B$ is $\mu$-layerwise computable and for all $n$ the function $f_n:=\frac{1}{n}\sum_{k=0}^{n-1}\chi_B\circ T^k$ is $\mu$-layerwise computable, uniformly in $n$. As a result, the $L_2$-norms of the functions $f_n-\mu(B)$ are all uniformly computable. Hence we can effectively find $n$ such that $\mu(B\cap \bigcap_{i\leq n}T^{-i}(A))<r\mu(B)$.

In the proof of Theorem~\ref{thm:generalized-kucera}, the computability of $T$ implied that the set $B\cap \bigcap_{i\leq n}T^{-i}(A)$ was effectively open. When $T$ is $\mu$-layerwise computable, the set $B\cap \bigcap_{i\leq n}T^{-i}(A)$ is effectively open on every $\mathcal{R}^\mu_k$. We end up with a test as in Lemma~\ref{lem:lay-test} enclosing $\bigcap_n T^{-n}(A)$, which implies the result.\qed
\end{proof}

In the same way, Theorems \ref{th:effective-birkhoff} and \ref{thm:effective-birkhoff} are true for computable probability spaces and for $\mu$-layerwise computable mappings $T$. In Theorem \ref{thm:effective-birkhoff}, the function $f$ can be assumed to be $\mu$-layerwise lower semicomputable.

\section{An application: the generalized van Lambalgen's theorem}\label{sec:gen-lambalgen}

The celebrated van Lambalgen theorem~\cite{vanLambalgen1987} asserts that in the probability space $\cs^2$ (pairs of binary sequences with independent uniformly distributed components) a pair $(\omega_0,\omega_1)$ is random if and only if $\omega_0$ is random and $\omega_1$ is $\omega_0$-random (random relative to the oracle $\omega_0$). This can be easily generalized to $k$-tuples: an element $(\omega_0,\omega_1,\ldots,\omega_{k-1})$ of $\cs^k$ is random if and only if $\omega_0$ is random and $\omega_i$ is $(\omega_0, \ldots, \omega_{i-1})$-random for all $i=1,2\ldots,k-1$. Can we generalize this statement to infinite sequences? Not completely: there exists an infinite sequence $(\omega_i)_{i \in \mathbb{N}}$ such that $\omega_0$ is random and $\omega_i$ is $(\omega_0, \ldots, \omega_{i-1})$-random for all $i \geq 1$; nevertheless, $(\omega_i)_{i\in\mathbb{N}}$ is non-random as an element of $\cs^\mathbb{N}$. To construct such an example, take a random sequence in $\cs^\mathbb{N}$ and then replace the first~$i$ bits of $\omega_i$ by zeros. 

Informally, in this example all $\omega_i$ are random, but their ``randomness deficiency'' increases with $i$, so the entire sequence $(\omega_i)$ is not random (in $\cs^\mathbb{N}$). K.~Miyabe~\cite{Miyabe-TA} has shown recently that one can overcome this difficulty allowing finitely many bit changes in each $\omega_i$ (number of changed bits may depend on $i$): 

\begin{theorem}[Miyabe]
Let $(\omega_i)_{i \in \mathbb{N}}$  be a sequence of elements of $\cs$ such that $\omega_0$ is random and  $\omega_i$ is $(\omega_0, \ldots, \omega_{i-1})$-random for all~$i \geq 1$. Then there exists a sequence $(\omega'_i)_{i \in \mathbb{N}}$ such that
\begin{itemize}
\item For every $i$ the sequence $\omega'_i$ is equal to $\omega_i$ except for a finite number of places.
\item The sequence $(\omega'_i)_{i \in \mathbb{N}}$ is a random element of $\cs^\mathbb{N}$. 
\end{itemize}
\end{theorem}

Informally, this result can be explained as follows: as we have seen (Theorem~\ref{thm:finite-change}), a change in finitely many places can decrease the randomness deficiency (starting from any non-random sequence, we get a sequence that is not covered by a first set of a Martin-L\"of test) and therefore can prevent ``accumulation'' of randomness deficiency. 

This informal explanation can be formalized and works not only for finite changes of bits but for any ergodic transformation. In fact, the results of this paper allow us to get a short proof of the following generalization of Miyabe's result (Miyabe's original proof used a different approach, namely martingale characterizations of randomness). We restrict ourselves to the uniform measure, but the same argument works for arbitrary computable measures.

\begin{theorem}\label{thm:gvl} Let $(\omega_i)_{i \in \mathbb{N}}$  be a sequence of elements of $\cs$ such that $\omega_0$ is random and  $\omega_i$ is $(\omega_0, \ldots, \omega_{i-1})$-random for all~$i \geq 1$. Let $T: \cs \rightarrow \cs$ be a computable bijective ergodic map. Then, there exists a sequence $(\omega'_i)_{i \in \mathbb{N}}$ such that
\begin{itemize}
\item For every~$i$, the sequence $\omega'_i$ is an element of the orbit of $\omega_i$ \textup(i.e., $\omega'_i=T^{n_i}(\omega_i)$ for some integer $n_i$\textup).
\item The sequence $(\omega'_i)_{i \in \mathbb{N}}$ is a random element of $\cs^\mathbb{N}$. 
\end{itemize}
\end{theorem}

\begin{proof} Let $U$ be the first level of a universal Martin-L\"of test on $\cs^\mathbb{N}$, with $\mu(U)\le 1/2$.  We will ensure that the sequence $(\omega_i')_{i\in\mathbb{N}}$ is outside $U$, and this guarantees its randomness.

Consider the set $V_0$ consisting of those $\alpha_0 \in \cs$ such that the section
\[
  U_{\alpha_0}=\left\{ (\alpha_1,\alpha_2,\ldots) \mid (\alpha_0,\alpha_1,\alpha_2,\ldots)\in U\right\}
  \]
has measure greater than $2/3$. The measure of $V_0$ is less than~$1$,  otherwise we would have $\mu(U) > 1/2$. It is easy to see that $V_0$ is an effectively open subset of $\cs$. Since $\omega_0$ is random, by Theorem~\ref{thm:generalized-kucera} there exists an integer $n_0$ such that $\omega'_0=T^{n_0}(\omega_0)$ is outside $V_0$. This $\omega'_0$ will be the first element of the sequence we are looking for.

Now we repeat the same procedure for $U_{\omega'_0}$ instead of $U$. Note that it is an open set of measure at most $2/3$, and, moreover, an effectively open set with respect to oracle $\omega'_0$. Since $\omega_0$ and $\omega'_0$ differ by a computable transformation, the set $U_{\omega'_0}$ is effectively open with oracle $\omega_0$.
We repeat the same argument (where $1/2$ and $2/3$ are replaced by $2/3$ and $3/4$ respectively) and conclude that there exists an integer $n_1$ such that the sequence $\omega'_1=T^{n_1}(\omega_1)$ has the following property: the set
	\[
U_{\omega'_0\omega'_1}=\left\{(\alpha_2,\alpha_3,\ldots)\mid (\omega'_0,\omega'_1,\alpha_2,\alpha_3,\ldots)\in U\right\}
	\]
has measure at most $3/4$. (Note that we need to use $\omega_0$-randomness of $\omega_1$, since we apply Theorem~\ref{thm:generalized-kucera} to an $\omega_0$-effectively open set.) 

At the next step we get $n_2$ and $\omega'_2=T^{(n_2)}\omega_2$ such that
	\[
U_{\omega'_0\omega'_1\omega'_2}=\left\{(\alpha_3,\alpha_4,\ldots)\mid (\omega'_0,\omega'_1,\omega'_2,\alpha_3,\alpha_4,\ldots)\in U\right\}
	\]
has measure at most $4/5$, etc.

Is it possible that the resulting sequence $(\omega'_0,\omega'_1,\omega'_2,\ldots)$ is covered by $U$? Since $U$ is open, it would be then covered by some interval in $U$. This interval may refer only to finitely many coordinates, so for some $m$ all sequences \[(\omega'_0,\omega'_1,\ldots,\omega'_{m-1},\alpha_m, \alpha_{m+1},\ldots)\] would belong to $U$ (for every $\alpha_m, \alpha_{m+1},\ldots$). However, this is impossible because our construction ensures that the measure of the set of all $(\alpha_m,\alpha_{m+1},\ldots)$ with this property is less than~$1$.
\qed
\end{proof}

Of course, the discussion of Section~\ref{subsec:cps} shows that Theorem~\ref{thm:gvl} can be extended to any computable probability space instead of the Cantor space, and to a layerwise computable ergodic map instead of a computable one. The details are left to the reader. 

\vspace{.5cm}

\textbf{Acknowledgements.} We would like to thank two anonymous referees for their very helpful comments and suggestions.

\bibliographystyle{alpha}	
\bibliography{effective_ergodic_theorems}
		
\end{document}